\documentclass[3p,12pt]{elsarticle}
\usepackage{t1enc}

\usepackage[latin2]{inputenc}
\usepackage{amssymb}
\usepackage{amsmath,amsthm}

\usepackage{mathrsfs}

\usepackage{latexsym}
\usepackage{epsfig}
\numberwithin{equation}{section}
\theoremstyle{plain}
\newtheorem{theorem}{Theorem}[section]
\newtheorem{lemma}{Lemma}[section]
\newtheorem{remark}{Remark}[section]

\newcommand{\ra}{\rightarrow}
\renewcommand{\b}{\beta}

\newcommand{\PP}{\text{P}}
\newcommand{\EE}{{\text{E}}}

\newcommand{\lle}{\,\,{\lesssim}\,\,}

\let\var\Var

\newcommand{\eps}{\varepsilon}

\renewcommand{\phi}{\varphi}

\newcommand{\given}{\,|\,}

\renewcommand{\th}{\theta}

\renewcommand{\a}{\alpha}

\newcommand{\lb}{\underline}

\begin{document}
\numberwithin{equation}{section}
\begin{frontmatter}

\title{\bf {Honest Bayesian confidence sets for the $L^2$-norm}}
\author[tue]{Botond Szab\'o \corref{cor1}}
\ead{b.szabo@tue.nl}
\author[lu]{Aad van der Vaart}
\ead{avdvaart@math.leidenuniv.nl}
\author[uva]{Harry van Zanten}
\ead{hvzanten@uva.nl}
\cortext[cor1]{Corresponding author}
\fntext[fn1]{Research supported by the Netherlands Organization for Scientific Research (NWO).
The research leading to these results has received funding from the European Research
Council under ERC Grant Agreement 320637.}
\fntext[fn2]{AMS 2000 subject classifications: Primary 62G15,62G05; secondary 62G20}
\address[tue]{Eindhoven University of Technology,
P.O. Box 513,
5600 MB Eindhoven, Netherlands}
\address[lu]{Leiden University, P.O. Box 9512,
2300 RA Leiden, Netherlands}
\address[uva]{University of Amsterdam, P.O. Box  94248,\\
1090 GE Amsterdam,
 Netherlands}

\begin{abstract}
We investigate the problem of constructing Bayesian credible sets that
are honest and adaptive for the $L^2$-loss over a scale of Sobolev classes with regularity ranging between $[D,2D]$,
for some given $D$ in the context of the signal-in-white-noise model. We consider a scale of prior distributions indexed by a regularity hyper-parameter and choose the hyper-parameter both by marginal likelihood empirical Bayes and by hierarchical Bayes method, respectively. Next we consider a ball centered around the corresponding posterior mean with prescribed posterior probability. We show by theory and examples that both the empirical Bayes and the hierarchical Bayes credible sets give
misleading, overconfident uncertainty quantification for certain oddly behaving
truth. Then we construct a new empirical Bayes method based on risk estimation, which provides the correct uncertainty quantification and optimal size.
\end{abstract}

\end{frontmatter}
\section{Introduction}
In Bayesian nonparametrics it is common to visualize the uncertainty of the posterior distribution by plotting the credible sets, i.e. sets accumulating a large fraction (typically $95\%$) of the posterior mass. These sets are often used in practice to quantify the uncertainty of a given estimate. They can be especially useful when the construction of confidence sets is not possible due to computational complexity or lack of theoretical results. However, the frequentist interpretation of the credible sets at the moment is rather unclear. In the present paper we investigate the asymptotic frequentist behaviour of Bayesian credible sets in the context of the signal-in-white-noise model.  We consider the sequence formulation
\begin{align}
X_i=\theta_{0,i}+\frac{1}{\sqrt{n}}Z_i,\quad\text{for all $i=1,2,...$}\label{def: model}
\end{align}
where $X=(X_1,X_2,...)$ is the observed infinite sequence, $Z_i$ are independent standard normal distributed random variables and $\theta_0=(\theta_{0,1},\theta_{0,2},..)$ is the unknown infinite dimensional parameter of interest. The signal-in-white-noise model is a relatively simple and tractable model but can be used at the same time as a platform to investigate more difficult statistical problems. For instance it is asymptotically equivalent with the regression model \cite{BrownLow}, and the density function estimation problem \cite{Nussbaum}. We expect our finding to generalize to the preceding models, however this does not follow straightforward from the asymptotic equivalence of the models, since we consider a particular method.

We assume that the true signal $\theta_0$ belongs to a collection of nested submodels $\cup_{\beta\in [D_1,D_2
]}\Theta^{\beta}$, for fixed $D_1$ and $D_2$, and $\Theta^{\beta_2}\subset\Theta^{\beta_1}$ for $\beta_1<\beta_2$. Considering any norm $\|\cdot\|$ a confidence set $\hat{C}_n$ corresponding this norm is called honest over $\cup_{\beta\in[D_1,D_2
]}\Theta^{\beta}$ if, for a level $\gamma>0$,
\begin{align*}
\liminf_{n\rightarrow\infty}\inf_{\theta_0\in \cup_{\beta\in[D_1,D_2]}\Theta^{\beta}}\PP_{\theta_0}(\theta_0\in\hat{C}_n)\geq 1-\gamma
\end{align*}
and rate adaptive if for all $\beta\in[D_1,D_2]$
\begin{align}
\liminf_{n\rightarrow\infty}\inf_{\theta_0\in \Theta^{\beta}}\PP_{\theta_0}(\|\hat{C}_n\|\leq C_\beta r_{n,\beta})\geq 1-\gamma,\label{def: honesty}
\end{align}
where $r_{n,\beta}$ denotes the minimax rate corresponding the norm $\|\cdot\|$ and class $\Theta^{\beta}$, and the constant $C_\beta$ depends only on the parameter $\beta$. Honesty (uniformity in $\theta_0$) is a relatively strong, but essential requirement. A pointwise (not uniform in $\theta_0$) confidence set has very limited applicability in practice, because in this case we know that for large enough $n$ the confidence set contains the true parameter with high probability, but what ``large enough'' means highly depends on the unknown parameter itself. Therefore pointwise asymptotic confidence sets provide only theoretical quantification of the uncertainty, in practice these sets are essentially uninformative.

It was shown in \cite{JudLam} and \cite{RobVaart} that the size of honest confidence sets over $\cup_{\beta\in[D_1,D_2]}\Theta^{\beta}$ is bounded below by the maximum of the minimax rate of estimation $r_{n,\beta}$ of $\theta\in\Theta^{\beta}$ and the minimax testing rate $\eps_{n,D_1}$ of $\theta\in\Theta^{\beta}$ against the alternative hypothesis $\{\theta\in\Theta^{D_1}:\, \|\theta-\Theta^{\beta}\|\geq\eps_{n,D_1}\}$. Usually the testing rate $\eps_{n,D_1}$ depends only on the larger submodel $\Theta_{D_1}$ and it is typically not bigger than the rate of estimation $r_{n,D_1}$ (over the submodel $\Theta^{D_1}$). Therefore for the existence of honest and adaptive confidence sets over a collection of submodels $\cup_{\beta\in[D_1,D_2]}\Theta^{\beta}$ we require that the minimax estimation rate over the smallest submodel $\Theta^{D_2}$ is not smaller than the minimax testing rate of $\theta_0\in \Theta^{D_2}$ against the largest submodel $\Theta^{D_1}$, i.e. $\eps_{n,D_1}\leq r_{n,D_2}$.

As a first example, consider the supremum-norm and the corresponding Sobolev ball
$$\Theta_{\infty}^{\beta}(M)=\{\theta\in\ell^2:\, \sup|\theta_i i^{\beta}|\leq M\}.$$
The minimax testing rate for $\theta_0\in \Theta_{\infty}^{\beta}(M)$, $\beta>D_1$, against the largest submodel $\Theta_{\infty}^{D_1}(M)$ is $\eps_{n,D_1}=(n/\log n)^{-D_1/(1+2D_1)}$, see (3.128) of \cite{ingster2}. Furthermore the minimax rate for estimating $\theta_0\in\Theta^{\beta}(M)$ is $r_{n,\beta}=(n/\log n)^{-\beta/(1+2\beta)}$, see (2.87) of \cite{ingster2}. Following from the lower bound introduced in \cite{JudLam} and \cite{RobVaart} the size of a honest confidence set over $\Theta_{\infty}^{\beta}(M)$, with $\beta>D_1$, is bounded from below by $(n/\log n)^{-D_1/(1+2D_1)}\gg (n/\log n)^{-\beta/(1+2\beta)}$. Therefore honesty and adaptivity can not hold at the same time for any choice of the parameters $D_1<D_2$. Similar results were concluded for the $L^{\infty}$-loss in various other settings, see for instance \cite{CaiLow}, \cite{GenWas} and \cite{Low}. To achieve honesty and adaptivity at the same time for $L^{\infty}$-loss some additional constraints have to be introduced, see for instance \cite{HenStark}, \cite{PicTri}, \cite{GineNickl}, \cite{Bull}.

However, the situation is rather different if we consider the $\ell^2$-loss and Sobolev balls
$$S^{\beta}(M)=\{\theta\in\ell^2:\, \sum_{i}\theta_i^2i^{2\beta}\leq M\}.$$
In this case the minimax testing rate for $\theta_0\in S^{\beta}(M)$, $\beta>D_1$, against the alternative hypothesis $S^{D_1}(M)$ is of order $n^{-D_1/(1/2+2D_1)}$; see Theorem 2.1 or 3.1 of \cite{ingster} or $(3.128)$ of \cite{ingster2}. Furthermore following from \cite{Pinsker} the minimax rate for estimating $\theta_0\in S^{\beta}(M)$ is a constant multiplier of $n^{-\beta/(1+2\beta)}$. These bounds suggest that the size of the honest confidence sets for $\beta\in[D_1,D_2]$ can be of the order $n^{-\beta/(1+2\beta)}\vee n^{-D_1/(1/2+2D_1)}$. For $\beta\leq 2D_1$ this means the minimax bound $n^{-\beta/(1+2\beta)}$ while for $\beta>2D_1$ a sub-optimal rate $n^{-D_1/(1/2+2D_1)}$. In our work we focus on the special case $D_2=2D_1$, where the size of the honest confidence sets are bounded from below by the minimax rate $n^{-\beta/(1+2\beta)}$, hence adaptation is possible. The existence of honest and adaptive confidence sets over $\cup_{\beta\in [D,2D]}S^{\beta}(M)$ were shown in \cite{RobVaart}, \cite{BullNickl}.

In this article we investigate whether Bayesian methods can reproduce the frequentist results and provide adaptive and honest confidence sets for the $\ell^2$-loss over the collection of Sobolev balls $\cup_{\beta\in [D,2D]}S^{\beta}(M)$. First we consider the empirical Bayes method based on marginal likelihood estimation of the regularity parameter and show that although the size of the credible sets achieve the optimal rate, the honesty requirement will not be fullfilled. We construct certain oddly behaving signals $\theta_0$ for which the marginal likelihood empirical Bayes method provides credible sets with coverage tending to zero, i.e. honesty (and also pointwise coverage) fails. A technical explanation of the preceding phenomenon relies on the bias-variance trade-off. In the marginal likelihood empirical Bayes method, for certain irregular signals $\theta_0$ the bias can dominate the posterior spread and the variance of the posterior mean, which leads to a coverage probability close to zero.

Next, we consider the hierarchical Bayes method with arbitrary hyper-prior distribution over $[D,2D]$. We show that the full Bayes method performs similarly to the marginal likelihood empirical Bayes method, in the sense that the hierarchical Bayes credible sets are not honest over $\cup_{\beta\in[D,2D]}S^{\beta}(M)$ (or in the case of a degenerate hyper-prior distribution at $D$ they are not adaptive). This result is perhaps not surprising in the light of \cite{SzVZ} and \cite{KSzVZ} where we have investigated the close relationship of these two techniques.

The negative results stated above show that the standard Bayesian techniques fail to achieve the frequentist limits. However, by modifying the empirical Bayes procedure one can construct optimally behaving credible sets. We introduce a new empirical Bayes method based on risk estimation, which provides honest credible sets and achieves adaptivity as well. As the first step of the technique we give an estimator for the squared bias and compute the posterior variance. Then we balance out these two quantities to get high coverage and at the same time optimal size for the credible sets. The method is based on the findings of \cite{RobVaart}.

The main message of this paper is that the choice of the statistical method has to be in accordance with the goal one wants to achieve. For instance if one evaluates the performance of the posterior mean with the mean integrated squared error then it could happen that the likelihood based procedures attain sub-optimal behaviour. A possible explanation of this phenomenon relies on that the mean integrated squared error is connected to the $\ell^2$-loss function, while the likelihood based methods, like the marginal likelihood empirical Bayes method and the hierarchical Bayes method, are related to the Kullback-Leibler divergence.

As mentioned above it is not possible to be honest and rate adaptive at the same time considering the $L^2$-loss function over the whole range of submodels $\cup_{\beta>D}S^{\beta}(M)$. However, by removing an asymptotically negligible set of signals from the collection of nested submodels $\cup_{\beta>D}S^{\beta}(M)$ the construction of honest and adaptive confidence sets is possible, see \cite{BullNickl}. We discuss briefly the extension of our risk based empirical Bayes method to cover this more general case as well.
%Extending our new, risk based empirical Bayes method into the direction of \cite{BullNickl} we expect to achieve similar statement. However, this matter is beyond the scope of the present paper, further research is necessary for a definite answer.
A slightly different direction to achieve honesty is to introduce additional constraints on the signals, similarly to the $L^{\infty}$-loss case. In \cite{SzVZ2} we introduced a new constraint, the ``polished tail'' condition, and showed that under this restriction the marginal likelihood empirical Bayes method produces honest confidence sets.

The remainder of the paper is organized as follows. In Section \ref{sec: Main} we describe in details the marginal likelihood empirical Bayes and full Bayes procedures and state that the credible sets based on the preceding techniques have frequentist coverage tending to zero for certain true signals. Then we introduce a new empirical Bayes technique, which provides adaptive and honest confidence sets. In Section \ref{sec: SimStud} we demonstrate both the negative and the positive findings by simulating the credible sets for an irregular, oddly behaving function. The proofs of the theorems of Section \ref{sec: Main} are given in Sections \ref{Sec: HierarchicalCounter} and \ref{sec: Coverage}. The proofs of additional auxiliary lemmas and theorems are deferred to Sections \ref{sec: estimator} and \ref{sec: Appendix}.

\subsection{Notation}
The $\ell^2$ norm of an element $\theta\in\ell^2$ is denoted by $\|\theta\|=(\sum_{i=1}^{\infty} \theta_i^2)^{1/2}$. For two real sequences $a_n$ and $b_n$ the notation $a_n\lesssim b_n$
means that $a_n/b_n$ is bounded, and $a_n\ll b_n$ that $a_n/b_n\rightarrow0$. For two real numbers $a$ and $b$, the notations $a\vee b$ and $a\wedge b$ denote the maximum and the minumum of the numbers, respectively. If $a$ denotes the empty set and $b$ is a real number, then both $a\vee b$ and $a\wedge b$ are taken to be equal to $b$. We denote the distribution of the infinite sequence $X$ corresponding to $\theta_0$ in $\eqref{def: model}$ by $\PP_{\theta_0}$ and the corresponding expectation and variance by $\EE_{\theta_0}$ and $\var_{\theta_0}$, respectively.

\section{Main result}\label{sec: Main}
\subsection{Model}\label{Sec: Model}
To make inference about the unknown sequence $\theta_0$ in the signal-in-white-noise model $\eqref{def: model}$ we endow it with a prior distribution
\begin{align}
\Pi_{\alpha}(\cdot)=\bigotimes_{i=1}^{\infty}N(0,i^{-1-2\alpha}),\label{prior}
\end{align}
where the parameter $\alpha>0$ denotes the regularity level. The corresponding posterior distribution $\theta|X\sim \Pi_{\a}(\cdot|X)$ can be easily computed
\begin{align}
\Pi_{\a}(\cdot|X)=\bigotimes_{i=1}^{\infty}N\Big(\frac{n}{i^{1+2\a}+n}X_i,\frac{1}{i^{1+2\alpha}+n}\Big).\label{eq: post}
\end{align}
The optimal choice of the hyper-parameter $\alpha=\beta$ leads to posterior contraction rate $n^{-\beta/(1+2\beta)}$, while for other choices we get sub-optimal contraction rates; see \cite{Bartek} and \cite{MR2471287}. Since the regularity parameter $\beta\in[D,2D]$ of the truth $\theta_0$ is usually not available one has to use a data driven method to choose $\alpha$.

\subsection{Marginal likelihood empirical Bayes method}\label{Sec: MargEB}

The first adaptive Bayes method we deal with is the marginal likelihood empirical Bayes method. In the Bayesian setting, described by the conditional
distributions $\theta\given\alpha \sim \Pi_\alpha$ and $X\given (\theta,
\alpha)\sim \otimes_{i}{N}(\theta_i,1/n)$, it holds that
\[
X \given \alpha \sim \bigotimes_{i=1}^{\infty}{N}(0,i^{-1-2\alpha}+ 1/n).
\]
The corresponding log-likelihood for $\alpha$ (relative to an infinite
product of $N(0,1/n)$-distributions) is given by
\begin{align}
\ell_n(\alpha)=-\frac{1}{2}\sum_{i=1}^{\infty}\Big( \log\Big(1+\frac{n}{i^{1+2\alpha}}\Big)-\frac{n^2} {i^{1+2\alpha}+n}X_i^2 \Big).\label{fn(alpha)}
\end{align}
We consider the maximum likelihood estimator, i.e. the estimator $\hat\alpha_n$ which maximizes the preceding marginal log-likelihood function in the interval $[D,2D]$, formally
\begin{align*}
\hat\a_n=\arg\max_{\a\in[D,2D]}\ell_n(\a).
\end{align*}
 Then the {\em empirical Bayes posterior} is defined as the random measure $\Pi_{\hat\alpha_n}(\cdot | X)$ obtained by
substituting $\hat\alpha_n$ for $\alpha$ in the posterior distribution given in $\eqref{eq: post}$, i.e.
\begin{align}
\Pi_{\hat\alpha_n}(A | X) = \Pi_{\alpha}(A | X) \Big|_{\alpha = \hat\alpha_n}\label{def: EBpost}
\end{align}
for measurable subsets $A \subset \ell^2$.

For fixed $\a>0$ the posterior distribution $\eqref{eq: post}$ is Gaussian with variance independent of the data, hence for given $\gamma\in(0,1)$ there exists a deterministic radius $r_{n,\gamma}(\a)$ such that the ball around the posterior mean $\hat\theta_\a=\hat\theta_{n,\a}$ contains $1-\gamma$ fraction of the posterior mass:
\begin{align}
\Pi_\a(\theta:\, \|\theta-\hat\theta_{n,\a}\|\leq r_{n,\gamma}(\a)|X)=1-\gamma.\label{def: radius}
\end{align}
In the empirical Bayes method we replace the fixed parameter $\a$ with the estimator $\hat\a_n$. We investigate the resulting credible ball (possibly) blown up by a constant multiplier $L>0$
\begin{align}
\hat{C}_n^{E}(L)=\{\theta:\, \|\theta-\hat\theta_{n,\a}\|\leq L r_{n,\gamma}(\hat\a_n) \}.\label{CredBal}
\end{align}

From \cite{SzVZ2} follows that the radius of the empirical Bayes credible sets $\eqref{CredBal}$ is rate adaptive over a collection of Sobolev balls $\cup_{\beta\in[D,2D]}S^{\beta}(M)$. Here we are interested wether the credible sets are also honest at the same time over $\cup_{\beta\in[D,2D]}S^{\beta}(M)$. Unfortunately the answer is negative to this question. By adapting Theorem 3.1 of \cite{SzVZ2} to the present setting we can show that for any regularity parameter $\beta\in[D,2D)$ there exists a sequence $\theta_0\in S^{\beta}(M)$ such that the coverage of the marginal likelihood empirical Bayes credible sets tends to zero along a subsequence.

\begin{theorem}\label{Thm: EBCounter}
 Take an arbitrary $\beta\in [D,2D)$, $\beta'\in[D,\beta)$ and $M>0$. Furthermore, take a sequence of positive integers $n_j$ such that $n_1\geq2$ and $n_j\gg n_{j-1}^{1+4D}$, let $K>0$ and define the sequence $\theta_0=(\theta_{0,1},\theta_{0,2},...)$ by
\begin{align}
\theta_{0,i}^2=\begin{cases}
Ki^{-1-2\beta},  & \text{if $n_j^{1/(1+2\beta)}\leq i<2n_j^{1/(1+2\beta)}$\, \text{for every}\, $j=1,2,...$}\\
0,  & \text{else}.\label{eq: Counter}
\end{cases}
\end{align}
Then the constant $K$ can be chosen such that $\theta_0\in S^{\beta'}(M)$ and for every $L>0$ the coverage of the credible set $\hat{C}_n^{E}(L)$ defined in $\eqref{CredBal}$ tends to zero, i.e. $\PP_{\theta_0}(\theta_0\in\hat{C}_{n_j}^{E}(L))\rightarrow 0$ as $j$ tends to infinity.
\end{theorem}

The proof of the theorem follows the line of the proof of Theorem 3.1 of \cite{SzVZ2} tailored to the present set up. The main difference between Theorem \ref{Thm: EBCounter} and Theorem 3.1 of \cite{SzVZ2} is that in the present setup we have the prior information that the true smoothness lies in the interval $\beta\in[D,2D]$. However, the maximizer of the marginal likelihood function can easily fall outside this interval, for instance in the case of $\eqref{eq: Counter}$. Therefore, in Theorem \ref{Thm: EBCounter} we are not concerned with the asymptotic performance of the global maximizer of the likelihood function like in Theorem 3.1 of \cite{SzVZ2}, but rather the local maximizer in the interval $[D,2D]$. We defer the proof to the appendix, Section \ref{sec: EBCounter}.

\subsection{Hierarchical Bayes}\label{sec: HB}

In the hierarchical, full Bayes method the hyper-parameter $\a$ in $\eqref{prior}$ is endowed with a hyper-prior distribution $\lambda$. Then the hierarchical prior distribution takes the form
\begin{align*}
\Pi(d\theta)=\int_{D}^{2D}\Pi_{\a}(d\theta)\lambda(d\a).
\end{align*}

We consider a ball around the hierarchical posterior mean $\hat\theta_n$ with radius $\hat{r}_{n,\gamma}$ such that it accumulates a fraction $1-\gamma$ of the posterior mass. In the construction of the hierarchical Bayes credible sets we introduce some additional flexibility by (possibly) blowing up the ball with a constant factor $L$
\begin{align}
\hat{C}_n^H(L)=\{\theta:\, \|\theta-\hat\theta_{n}\|\leq L \hat{r}_{n,\gamma} \}.\label{CredBal3}
\end{align}

Similarly to the marginal likelihood empirical Bayes method the hierarchical Bayes method also chooses for certain oddly behaving sequences $\theta_0$ a sub-optimal hyper-parameter by concentrating the hyper-posterior distribution around it.  Therefore, the hierarchical Bayes credible sets $\eqref{CredBal3}$ are not honest and/or have sub-optimal size.

\begin{theorem}\label{Thm: HierarchicalCounter}
For any choice of the hyper-prior $\lambda$ there exist $\beta\in[D,2D]$ and $\theta_0\in S^{\beta}(M)$ (for arbitrary $M>0$) such that for every $L>0$ the hierarchical Bayes credible set defined in $\eqref{CredBal3}$, has sub-optimal size $\|\hat{C}_n^{H}(L)\|\gg n^{-\beta/(1+2\beta)}$ with $\PP_{\theta_0}$-probability tending to one and/or has frequentist coverage tending to zero along a subsequence.
\end{theorem}

\begin{proof}
See Section \ref{Sec: HierarchicalCounter}.
\end{proof}

In our setup the main difference between the marginal likelihood empirical Bayes technique and the hierarchical Bayes method is that the former is conditionally Gaussian given the observations while the latter has much more complicated distribution. In the proof of Theorem \ref{Thm: HierarchicalCounter} we use the results on the asymptotic behaviour of the maximum likelihood estimator $\hat\a_n$, but the different nature of the posterior distributions requires separate analysis.

\subsection{Risk based empirical Bayes method}\label{Sec: ExtEB}

The main problem with the marginal likelihood empirical Bayes method is that the estimator maximizes the likelihood function instead of minimizing the estimated mean squared error of the posterior mean. This could cause a wrong bias-variance trade-off and therefore bad coverage result. In the present section we aim to correct this problem and give another estimator for the hyper-parameter which provides adaptive and honest empirical Bayes credible sets over $\cup_{\beta\in[D,2D]}S^{\beta}(M)$. The idea of our estimator relies on the technique introduced in \cite{RobVaart}.

First we give an estimator for the squared norm of the bias $B_n^2(\a;\theta_0)=\|\theta_0-E_{\theta_0}\hat\theta_{n,\a}\|^2=\sum_{i=1}^{\infty} i^{2+4\a}\theta_{0,i}^2/(i^{1+2\a}+n)^2$ with fixed hyper-parameter $\a$:
\begin{align*}
\hat{B}_{n,k_n}^2(\a)=\sum_{i=1}^{k_n}\frac{i^{2+4\a}}{(i^{1+2\a}+n)^2} (X_i^2-\frac{1}{n}),
\end{align*}
where the sequence $k_n$ will be specified later.
One can observe that the expected value of the preceding estimator is
\begin{align}
B_{n,k_n}^2(\a;\theta_0):=\EE_{\theta_0}\hat{B}_{n,k_n}^2(\a)=\sum_{i=1}^{k_n}i^{2+4\a}\theta_{0,i}^2/(i^{1+2\a}+n)^2.\label{def: Biask}
\end{align}
Hence for $\theta_0\in S^{\beta}(M)$ the bias of the estimator $\hat{B}_{n,k_n}^2(\a)$ is bounded above by
\begin{align}
\sum_{i=k_n+1 }^{\infty} i^{2+4\a}\theta_{0,i}^2/(i^{1+2\a}+n)^2
\leq (k_n+1)^{-2\beta}\sum_{i=k_n+1}^{\infty} i^{2\beta}\theta_{0,i}^2
\leq M k_n^{-2\beta}.\label{eq: BiasLK}
\end{align}
For the choice $k_n=n^{1/(1/2+2D)}$ the right hand side of the previous display is further bounded from above by $Mn^{-4D/(1+4D)}$ for any $\beta\in[D,2D]$.

With the help of the preceding estimator $\hat{B}_{n,k_n}(\a)$ we define
\begin{align}
\tilde\a_n=\inf\{\a\geq D:\,\hat{B}_{n,k_n}(\a)\geq C_1n^{-\a/(1+2\a)}\}\wedge ((2D-C_0/\log n)\vee D),\label{def: estimator}
\end{align}
with
\begin{align}
 C_1>0 \quad\text{and}\quad C_0=(1+4D)^2\log(25C_1^{-2}\gamma^{-1})/2\vee 0.\label{def: C0}
\end{align}
The parameter $C_1$ controls the degree of under smoothing; a smaller choice for the parameter $C_1$ results in a smaller estimator for the regularity parameter $\tilde\a_n$. The parameter $C_0$ controls the behaviour of the estimator close to the upper bound $2D$. It is a monotonically decreasing function of $C_1$.

We use this estimator in the empirical Bayes procedure and define the risk based empirical Bayes posterior by substituting $\tilde\a_n$ defined in $\eqref{def: estimator}$ for $\a$ in the posterior $\eqref{eq: post}$. The risk based empirical Bayes credible sets are constructed as
\begin{align}
\hat{C}_n^{R}(L):=\{ \theta: \|\theta-\hat{\theta}_{n,\tilde\alpha_n}\|<Lr_{n,\gamma}(\tilde{\alpha}_n)\},\label{CredBal2}
\end{align}
where $L$ is a scaling parameter, $r_{n,\gamma}$ is the radius of a $1-\gamma$ credible ball for fixed $\a$ defined in $\eqref{def: radius}$ and $\tilde{\alpha}_n$ is the new estimator of the hyper-parameter. We show that the credible sets defined in $\eqref{CredBal2}$ are honest over the collection of Sobolev balls $\cup_{\beta\in[D,2D]}S^{\beta}(M)$ and rate adaptive over $S^{\beta}(M)$ for all $\beta\in[D,2D]$.

\begin{theorem}\label{thm: Coverage}
For arbitrary positive parameters $D,M, C_1$ and $\gamma$ the credible sets defined in $\eqref{CredBal2}$ with the constant factor $L\geq \sqrt{8(1+3^{1+4D})}(\sqrt{6}+\sqrt{2(C_1^2+ M)})$ are honest
\begin{align}
\liminf_{n\rightarrow\infty}\inf_{\theta_0\in \cup_{\beta\in[D,2D]}S^{\beta}(M)} \PP_{\theta_0}\big(\theta_0\in \hat{C}_n^{R}(L)\big)\geq 1-\gamma.\label{eq: PostCov}
\end{align}
Furthermore the radius of the credible set is rate adaptive, i.e. for all $\beta\in[D,2D]$
\begin{align}
\liminf_{n\rightarrow\infty}\inf_{\theta_0\in S^{\beta}(M)}\PP_{\theta_0}\big(r_{n,\gamma}(\hat{\alpha}_n)\leq K  n^{-\beta/(1+2\beta)}\big)\geq 1-\gamma\label{eq: adaptivity}
\end{align}
with
$$K=\sqrt{3+2/D}\exp\Big(\frac{2[(1/2+\beta)\log(2M/C_1^2)\vee C_0]}{(1/2+2D)^2}\Big).$$
\end{theorem}

\begin{proof}
See Section~\ref{sec: Coverage}.
\end{proof}

The scaling parameter $L$, similarly to $C_1$ (given in $\eqref{def: estimator}$), also controls the degree of under smoothing. It can be seen that a smaller choice of the parameter $C_1$ results in a smaller value for the scaling factor $L$. At the same time the radius of the credible ball $\eqref{eq: adaptivity}$ is monotonically increasing as $C_1$ goes to zero.

\begin{remark}\label{rem: toosmooth}
From the definition of $\tilde\a_n$ one can see that for regularity parameter $\beta>2D$ the estimator of the hyperparameter $\tilde\a_n$ ``undersmooths'' the truth, i.e. chooses smaller regularity parameter ($\tilde\a_n\leq 2D$) than the true regularity. Therefore the size of the credible set will be sub-optimal, but the coverage statement holds (actually we get asymptotically conservative coverage one in this case).
\end{remark}

\subsection{Extension to the case $D_2>2D_1$}\label{sec: Extension}
A natural question is the performance of Bayesian procedures for $D_2>2D_1$, i.e. without the assumption $\beta\in[D_1,2D_1]$. As we have discussed it already in the introduction, the construction of adaptive and honest confidence sets is impossible in this case. Therefore there is no hope for a Bayesian based procedure to provide confidence sets with good coverage properties and rate adaptive size. However, by introducing some additional constraints, like in \cite{BullNickl}, the construction of adaptive and honest confidence sets is possible. We show that by slightly adapting our Bayesian based method (introduced in Section \ref{Sec: ExtEB}) we can achieve the frequentist limit.

For given $0<D_1<D_2<\infty$ as a first step we define the grid
$$ B= \{\beta_m\}_{m=1}^{N}=\{D_1,2D_1,4D_1,...,2^{N-1}D_1\},$$
for $2^{(N-1)}D_1\leq D_2<2^{N}D_1$ and with the help of the grid we introduce the notation, for $\beta\in B\backslash\{\beta_{N}\}$,
\begin{align*}
\tilde{S}\big(\beta,\rho,M\big)=\tilde{S}\big(\beta, B,\rho_n(\beta),M\big)=\{\theta\in S^{\beta}(M):\, \|\theta-S^{\alpha}(M)\|\geq \rho_n(\beta),\,\forall \a>\beta,\,\a\in B\},
\end{align*}
denoting the collection of $\beta$-regular signals which are at least
$$\rho_n(\beta):=n^{-\beta/(1/2+2\beta)}$$
far away from any Sobolev ball $S^{\a}(M)$ (for $\a>\beta$ and $\a\in B$). Then we can take the union of such sets
$$\mathcal{P}_n(M, B)=S^{\beta_{N}}(M)\bigcup\Big(\bigcup_{\beta\in  B\backslash\{\beta_N\}}\tilde{S}(\beta, \rho,M)\Big).$$
One can observe that the above set tends to $S^{D_1}(M)$ as $n$ goes to infinity, hence the ``left out signals'' asymptotically vanish.
Furthermore from Theorem 1 and Theorem 5 of \cite{BullNickl} follows that $\mathcal{P}_n(M, B)$ is the largest set on which the construction of adaptive and honest confidence sets is possible in $\ell_2$-norm (for known radius parameter $M>0$).

Next we introduce a slight modification (following the technique in \cite{BullNickl}) of our risk based empirical Bayes method which allows us to reproduce the frequentist results. Using the adapted version of the test given in $(17)$ of \cite{BullNickl} to our setting we can give an estimator $\hat{\beta}_n\in B$ such that for any $\theta_0\in \mathcal{P}_n(M, B)$ we have
$$\PP_{\theta_0}(\theta_0\in \cup_{\beta\in[\hat\beta_n,2\hat\beta_n]}S^{\beta}(M))\geq1-\gamma/2.$$
Then we modify the estimator $\tilde\a_n$ given in $\eqref{def: estimator}$ by plugging in $\hat{k}_n=n^{1/(1/2+2\hat{\beta}_n)}$. Along the lines of the proof of Theorem 5 of \cite{BullNickl} and Theorem \ref{thm: Coverage} it can be shown that the so defined extended risk based empirical Bayes method provides adaptive and honest confidence sets over $\mathcal{P}_n(M, B)$.

\subsection{Discussion}
The marginal likelihood empirical Bayes method and hierarchical Bayes method are closely related. They differ only in that the empirical Bayes method takes the maximizer of the marginal Bayesian likelihood whereas the hierarchical Bayes approach equips this marginal likelihood with a hyper-prior. Therefore the sub-optimal behaviour of the marginal likelihood empirical Bayes approach (at least in our setting) leads to a sub-optimal behaviour of the hierarchical Bayes technique. The situation is not as  bad as it looks at first sight, by appropriate choice of the hyper-parameter the construction of adaptive and honest credible sets is possible over $\cup_{\beta\in[D,2D]}S^{\beta}(M)$. However, these credible sets have no close relation with the full Bayes method. The applied estimator is based on balancing out the bias and variance, and therefore it is substantially different from the marginal likelihood empirical Bayes method and hence from the hierarchical (full) Bayes method.

Another important question is the practical applicability of the derived risk based empirical Bayes method. First of all we note that similarly to the marginal likelihood empirical Bayes method it is computationally substantially faster than the hierarchical Bayes method (except perhaps in case of conjugate hyper-priors). Furthermore following from Remark~\ref{rem: toosmooth} the sets $\hat{C}_n^{R}$ can be used in practice for uncertainty quantification even in the case $\beta>2D$ in the sense that they contain the truth with high probability for large enough $n$ (they just do not achieve the optimal size which is perhaps the smaller problem). Of course if one wants to achieve full adaptation then one can apply the extended method of Section~\ref{sec: Extension} at the price of throwing out certain (asymptotically vanishing) subsets.
For large values of $n$  the constant multipliers do not play an important role, therefore the choice of the parameter $C_1$ in $\eqref{def: C0}$ and as a consequence $C_0$ and $L$ has no real effect on the behaviour of the credible sets. For small sample size the parameter $C_1$ plays a crucial role. A smaller choice of $C_1$ leads to better coverage results but a larger set size, while a larger $C_1$ results in a worse coverage level but smaller set size. Therefore (depending on the noise level) we recommend to use smaller values of $C_1$.

Finally, we note that our results (after slightly adapting the proofs) also hold for other regularity classes, for instance hyperrectangles. However, in the present paper we consider only Sobolev balls for better tractability and for better connection with the frequentist literature.

\section{Simulation study}\label{sec: SimStud}

To illustrate our findings we consider the functional formulation of the signal-in-white-noise model
\begin{align*}
X_t=\int_{0}^t \theta_0(s)ds+(1/\sqrt{n})W_t,\quad t\in[0,1],
\end{align*}
where $X_t$ is the noise observation, $\theta_0$ the unknown function of interest and $W_t$ denotes the Wiener process. Then we simulate data from this model for $\theta_0$ given by its Fourier coefficients
\begin{align*}
\theta_{0,i}=\begin{cases}
\sin(i)10^{-1.7}, & \text{if $i=10..20$,}\\
3\sin(i)100^{-1.7},  & \text{if $i=100..150$,}\\
i^{-1.2}, & \text{if $i=4^{4^j}...2\times4^{4^j}$ for any $j=2,...$,}\\
0,  & \text{else}.
\end{cases}
\end{align*}
with respect to the eigen basis $\phi_i(t)=\sqrt{2}\cos((i-1/2)\pi t)$. We note that the function $\theta_0$ is contained in the Sobolev ball $S^{\beta'}(M)$ with any $\beta'<\beta:=1.2$ and sufficiently large $M$. Furthermore we assume the prior knowledge that $\beta\in[1,2]$.

\begin{figure}[htbp]
  \centering
   \includegraphics[width=14cm]{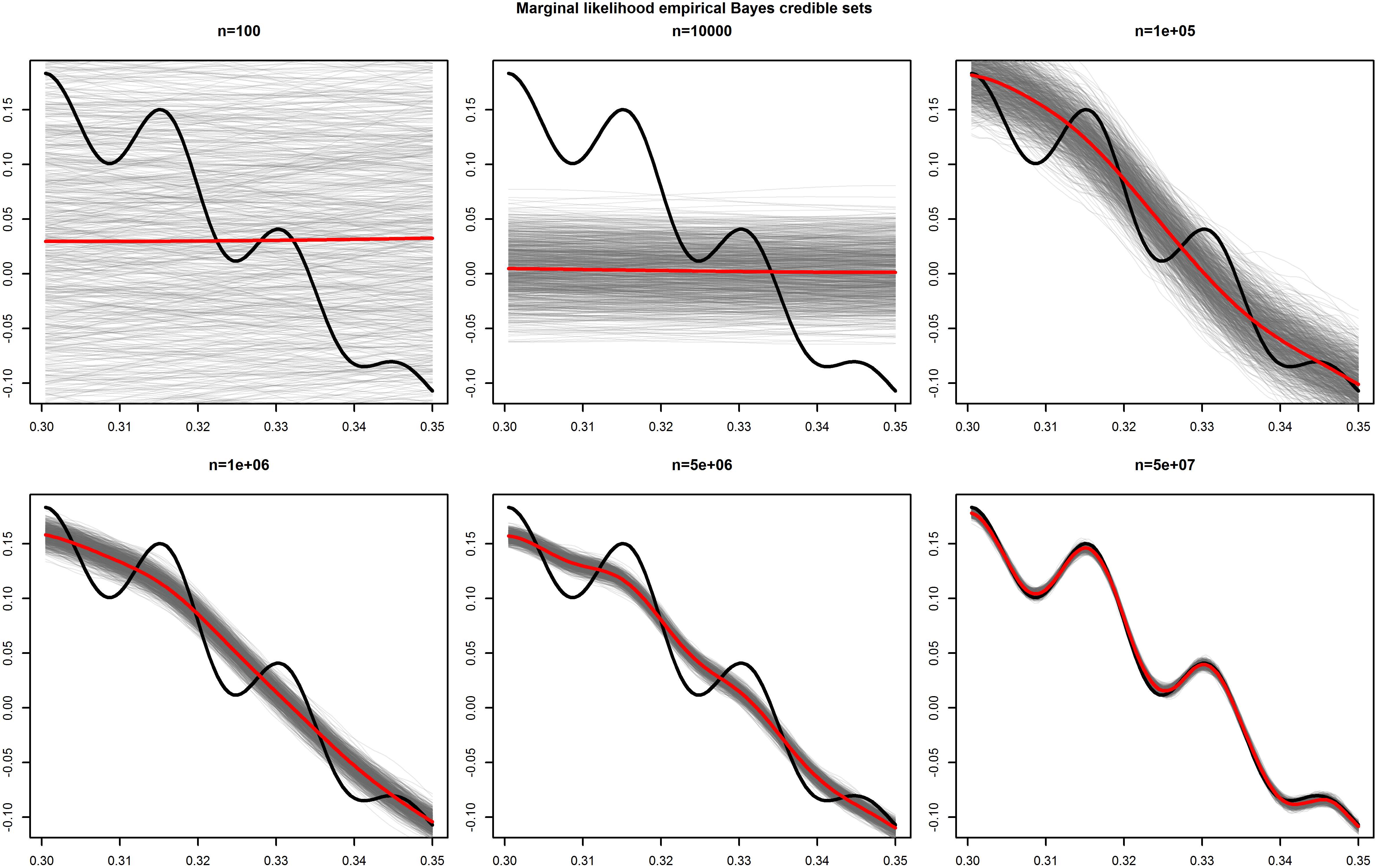}
 \caption{Marginal likelihood empirical Bayes credible sets for unregular function. The true function is drawn in black, the posterior mean in red and the credible set in grey. We have $n=10^2,10^4,10^5,10^6,5\times10^6$ and $5\times10^7$.}
\label{fig: eb}
\end{figure}

In Figure \ref{fig: eb} we visualize the $95\%$ marginal likelihood empirical Bayes credible sets defined in $\eqref{CredBal}$ for $n=10^2,10^4,10^5,10^6,5\times10^6$ and $5\times10^7$, respectively. We note that in Figure \ref{fig: eb} and also in the following figures we plot the zoomed in version of the functions to the interval $[0.3,0.35]$. The true function is drawn in black, the posterior mean in blue and the grey area is the collection of the $95\%$ closest out of $800$ draws to the posterior mean from the posterior distribution, which gives a good indication of the $95\%$-credible set. In the simulation study we do not blow up the credible sets by a constant factor, i.e. we worked with $L=1$ in all three cases. The figure indicates that the credible set fails to cover the true function along a subsequence.
\begin{figure}[htbp]
  \centering
   \includegraphics[width=14cm,height=8cm]{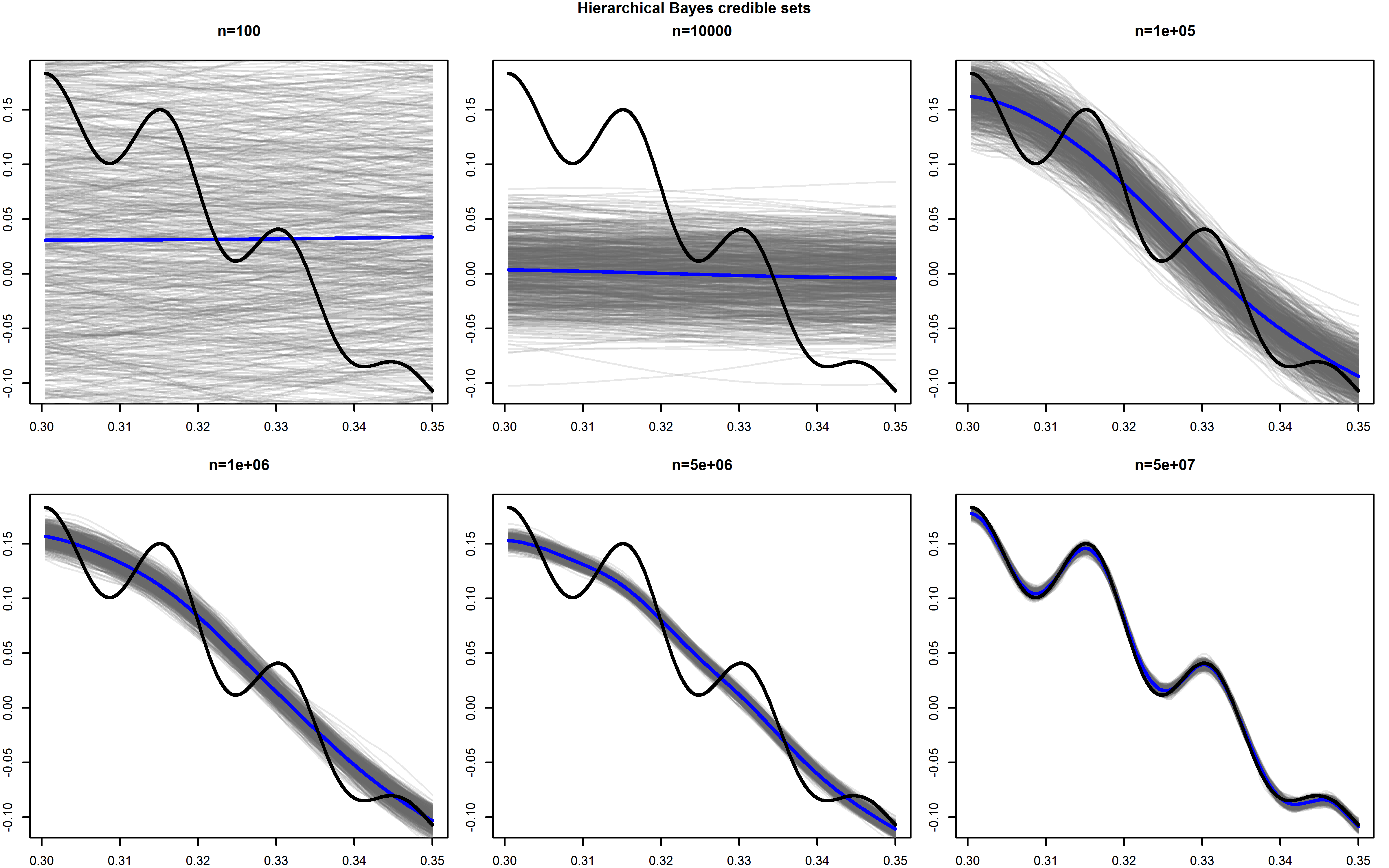}
 \caption{Hierarchical Bayes credible sets. The true function is drawn in black, the posterior mean in blue and the credible sets in grey. We have $n=10^2,10^4,10^5,10^6,5\times10^6$ and $5\times10^7$.}
\label{fig: hb}
\end{figure}

To demonstrate that the hierarchical Bayes method has bad coverage performance along a subsequence we plot the credible set given in $\eqref{CredBal3}$ with uniform hyper-prior $\lambda$ on $[D,2D]$ in Figure \ref{fig: hb}. Since the posterior distribution can not be computed explicitly we used an MCMC method generating draws from the hierarchical posterior distribution. As a first step we truncate the infinite dimensional vector $\theta_0$ to its first $N=n^{1/(1/2+2D)}$ coefficients $\theta_0^{N}$. This way the approximation $\|\theta-\theta^{N}\|$ is of smaller order than the contraction rate. Then we apply a Metropolis within Gibbs sampling algorithm for sampling draws $(\alpha,\theta^{N})$ from the posterior distribution. We alternate between draws $\theta^{N}|(\alpha,X)$ which can be done explicitly and $\alpha|(\theta^{N},X)$. For the latter we use standard Metropolis-Hastings algorithm with uniform proposal distribution over $[D,2D]$ for $\alpha$. We choose a $3200$ iteration long burn in period and then sample $800$ draws. The implementation is straightforward hence we omit further description of the algorithm. One can see that similarly to the marginal likelihood empirical Bayes method the true function (in black) is not included in the credible sets for $n=10^4,10^6$ and $5\times10^6$.

\begin{figure}
\centerline{
    \includegraphics[width=14cm]{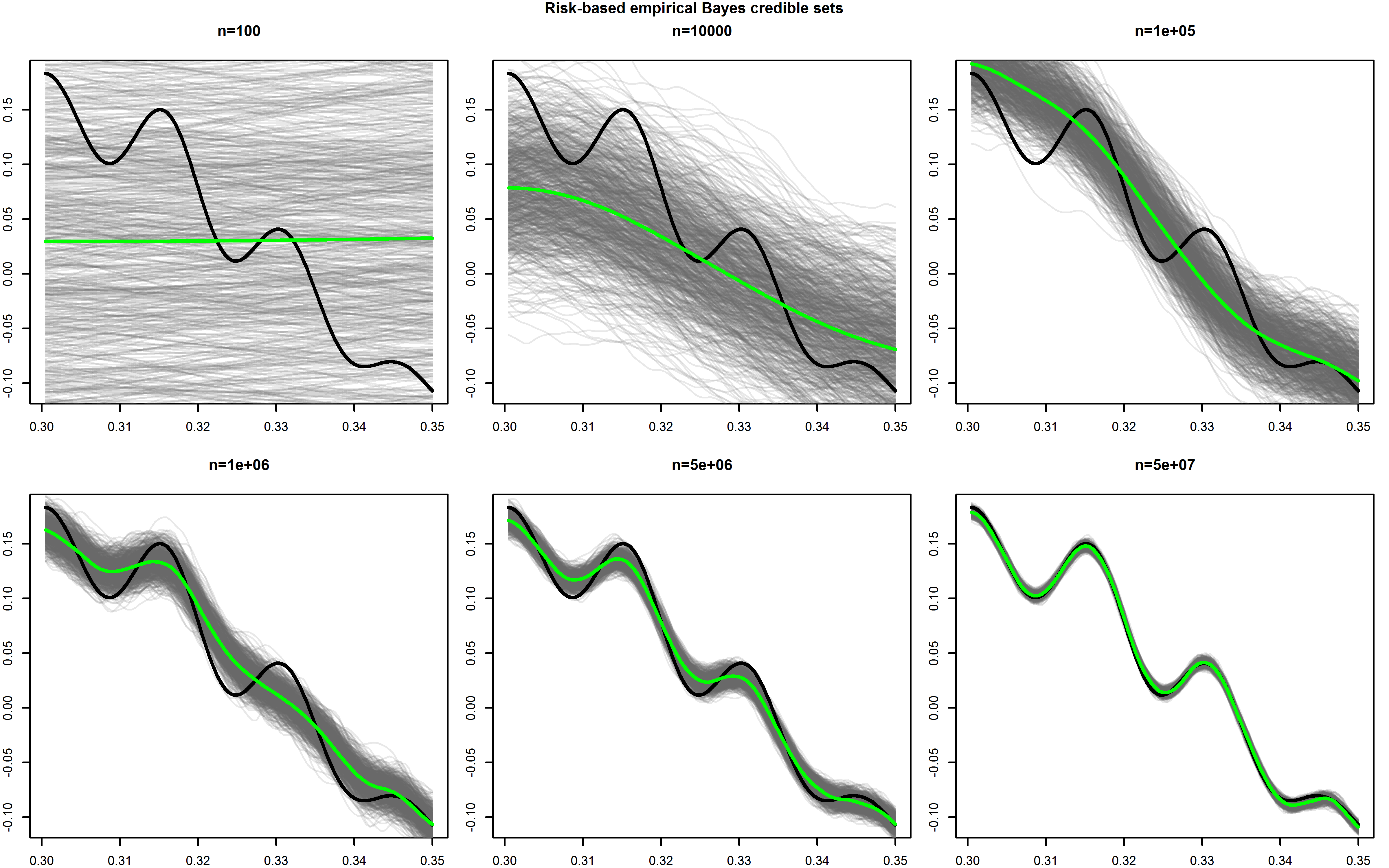}}
 \caption{Risk based empirical Bayes credible sets. The true function is drawn in black, the posterior mean in green and the credible sets in grey. We have $n=10^2,10^4,10^5,10^6,5\times10^6$ and $5\times10^7$.}
\label{fig: ieb}
\end{figure}

Finally in an attempt to illustrate the better coverage property of the risk based empirical Bayes credible sets we use the same simulated data as in the preceding two cases and plot in Figure \ref{fig: ieb} the corresponding $95\%$-credible set. In the particular example we choose the parameter $C_1$ (given in $\eqref{def: estimator}$) to be $C_1=1/3$, hence the constant $C_0$ defined in $\eqref{def: C0}$ is $105.1$. In the example we take $C_0$ to be zero, because this way we avoid very conservative credible sets, i.e. credible sets from overly under smoothed posterior distributions ($\tilde\a_n=D$). It can be clearly seen that the grey areas cover the true function $\theta_0$ plotted in black even in the critical cases $n=10^4,10^6$ and $5\times10^6$ where the other two Bayesian procedures failed.

\section{Proof of Theorem \ref{Thm: HierarchicalCounter}}\label{Sec: HierarchicalCounter}
In the case that the hyper prior $\lambda$ is the Dirac measure concentrated at $\a_0\in[D,2D]$, the hierarchical posterior distribution will be $\Pi_{\a_0}(\cdot|\,X)$ see $\eqref{eq: post}$. By choosing $\beta\neq\a_0$ and $\theta_0\in S^{\beta}(M)$ the posterior distribution $\Pi_{\a_0}(\cdot|\,X)$ achieves a sub-optimal contraction rate around $\theta_0$ following from \cite{Castillo}, hence we get sub-optimal size for the credible set and/or coverage tending to zero.

It remains to deal with non-degenerate hyper-prior distributions $\lambda$, for which there exist some constants $c>0$ and $D^*\in (D,2D)$ such that $\lambda(\a\in[D^*,2D])>c$. One can see that for any given $n$ there exists a parameter $D'=D'(n)\in[D_1,2D-1/(\log n)]$ such that $\lambda([D'+1/(2\log n),D'+1/(\log n)])>c/(2\log n)$. Furthermore take any $\beta\in(D,D^*)$ and define $\theta_0$ as in $\eqref{eq: Counter}$ (with small enough constant $K$ as in Theorem~\ref{Thm: EBCounter}). Finally let us introduce the notation
\begin{equation}\label{def: BiasVar}
W_n(\a) = \hat\th_\a - E_{\th_0}\hat\th_{n,\a}, \qquad  \text{ and }\qquad
B_n(\a;\th_0) = E_{\th_0}\hat\th_{n,\a} - \th_0,
\end{equation}
for the centered posterior mean and bias, respectively.

From triangle inequality we have
\begin{align*}
&\PP_{\theta_0}\big(\|\hat\theta_{n_j}-\theta_0\|\leq L \hat{r}_{n_j,\gamma}\big)\leq\\
&\quad \PP_{\theta_0}\Big(\|B_{n_j}(D',\theta_0)\|\leq \|\hat\theta_{n_j}-\hat\theta_{n_j,D'}\|
+\|W_{n_j}(D')\|
+ L \hat{r}_{n_j,\gamma}\Big),
\end{align*}
To prove that the right hand side of the preceding display tends to zero it is sufficient to show that there exist positive constants $A_1,A_2,A_3,A_4$ such that
\begin{align}
\|B_{n_j}(D';\theta_0)\|^2\geq A_1 n_j^{-2\beta/(1+2\beta)},\label{eq: LBBias}\\
\PP_{\theta_0}\big(\|W_{n_j}(D')\|^2\leq A_2 n_j^{-2D'/(1+2D')}\big)\rightarrow1,\label{eq: CentPostMean}\\
\PP_{\theta_0}\big(\sup_{\a\in[D',2D]}\|\hat\theta_{n_j}-\hat\theta_{n_j,\alpha}\|^2\leq A_3 n_j^{-2D'/(1+2D')}\big)\rightarrow1,\label{eq: help4}\\
\PP_{\theta_0}\big(\hat{r}_{n_j,\gamma}^2\leq A_4 n_j^{-2D'/(1+2D')}\big)\rightarrow1,\label{Ass: Help3}
\end{align}
since $n_j^{-2D'/(1+2D')}\ll n_j^{-2\beta/(1+2\beta)}$.
We note that $\eqref{eq: help4}$ can be replaced by a weaker assertion but for the proof of $\eqref{Ass: Help3}$ we need it in the present form. The proof of assertion $\eqref{eq: CentPostMean}$ follows from $(5.9)$ of \cite{SzVZ2} with $\a=D'$ and $p=0$, i.e.
\begin{align}
\inf_{\theta_0\in\ell^2}\PP_{\theta_0}(\sup_{\a\geq\a_1}\|\hat{\theta}_{n,\a}-\EE_{\theta_0}\hat\theta_{n,\a}\|^2\leq 6n^{-2\a_1/(1+2\a_1)})\rightarrow 1.\label{eq: UBw}
\end{align}
For $\eqref{eq: LBBias}$
we have following from the explicit expression for $\hat\theta_{D'}$ in $\eqref{eq: post}$, the inequality $D'>\beta$ and the definition of $\theta_0$ given in $\eqref{eq: Counter}$ with the notation $N_{j}=n_j^{1/(1+2\b)}$ that
\begin{align}
\bigl\|B_{n_j}(D';\th_0)\bigr\|^2&\ge
\sum_{N_j \le i <2N_j}\frac{i^{2+4D'}\th_{0,i}^2}{(i^{1+2D'}+n_j)^2}
\ge \frac 1{4}\sum_{N_j\le i <2N_j}\th_{0,i}^2\gtrsim n_j^{-2\b/(1+2\b)}.\label{eq: BiasLB}
\end{align}
For the proofs of assertions $\eqref{eq: help4}$ and $\eqref{Ass: Help3}$ we refer to Sections \ref{Sec: Help4} and \ref{Sec: Help3}, respectively.

\subsection{Proof of assertion $\eqref{eq: help4}$}\label{Sec: Help4}
By Jensen's inequality, Fubini's theorem and triangle inequality one can obtain that
\begin{align}
&\sup_{\a_1\in[D',2D]}\| \hat\theta_{n_j}-  \hat\theta_{n_j,\a_1} \|^2= \sup_{\a_1\in[D',2D]}\| \int_{D}^{2D}\big(   \hat\theta_{n_j,\a}-   \hat\theta_{n_j,\a_1}\big)\lambda(d\a|X)\|^2\nonumber\\
&\quad\leq \sup_{\a_1\in[D',2D]}\sum_{i=1}^{\infty}\int_{D}^{2D}\big(\hat\theta_{n_j,\a,i}-  \hat\theta_{n_j,\a_1,i}\big)^2\lambda(d\a|X)\nonumber\\
&\quad\leq\sup_{\a_1,\a_2\in[D',2D]}\|\hat\theta_{n_j,\a_1}-  \hat\theta_{n_j,\a_2}\|^2\lambda(\a\in[D',2D]|X)\nonumber\\
&\quad\quad+
\sup_{\a_1,\a_2\in[D,2D]}\|\hat\theta_{n_j,\a_1}-  \hat\theta_{n_j,\a_2}\|^2\lambda(\a\in[D,D']|X). \label{eq: help2}
\end{align}
Next we show that both terms on the right hand side of $\eqref{eq: help2}$ are bounded above by constant times $n_j^{-2D'/(1+2D')}$ with $\PP_{\theta_0}$-probability tending to one.

Starting with the first term, we use the trivial bound 1 for the hyper-posterior probability. Furthermore we have
\begin{align}
\sup_{\a_1,\a_2\in[D',2D]}\| \hat\theta_{n_j,\a_1}-  \hat\theta_{n_j,\a_2}\|^2&\leq \sup_{\a_1,\a_2>\in[D',2D]}\|\EE_{\theta_0} \hat\theta_{n_j,\a_1}-  \EE_{\theta_0}\hat\theta_{n_j,\a_2}\|^2\nonumber\\
&\quad+2\sup_{\a\in[D',2D]}\| \hat\theta_{n_j,\a}- \EE_{\theta_0} \hat\theta_{n_j,\a}\|^2.\label{eq: seged}
\end{align}

From $\eqref{eq: UBw}$ with $\a_1=D'$ the second term on the right hand side of $\eqref{eq: seged}$ is bounded above by a multiple of $n_j^{-2D'/(1+2D')}$ with $\PP_{\theta_0}$-probability tending to one.
The first term on the right hand side of $\eqref{eq: seged}$ can be written as $\sup_{\a_1,\a_2\in[D',2D]}\sum_{i=1}^{\infty}(f_i(\a_1)-f_i(\a_2))^2$ for $f_i(\a)=n_j\theta_{0,i}/(n_j+i^{1+2\a})$. The derivative of $f_i(\a)$ is $-2\log (i)i^{1+2\a}n_j\theta_{0,i}/(n_j+i^{1+2\a})^2$. Writing the difference as the integral of $f_i'(\a)$, applying Cauchy-Schwarz inequality to its squares and then interchanging the sum and the integral we get that
\begin{align}
\sum_{i=1}^{\infty}(f_i(\a_1)-f_i(\a_2))^2&= \sum_{i=1}^{\infty}\Big(\int_{\a_1}^{\a_2}f_i'(\a)d\a\Big)^2
\leq \sum_{i=1}^{\infty}(\a_2-\a_1)\int_{\a_1}^{\a_2}(f_i'(\a))^2d\a\nonumber\\
&= (\a_2-\a_1)\int_{\a_1}^{\a_2}\sum_{i=1}^{\infty}(f_i'(\a))^2d\a\leq (\a_2-\a_1)^2 \sup_{\a\in[\a_1,\a_2]}\sum_{i=1}^{\infty}(f_i'(\a))^2\nonumber\\
&\leq 4D^2 \sup_{\a\in[D',2D]}\sum_{i=1}^{\infty}\frac{n_j^2\theta_{0,i}^2(\log i)^2i^{2+4\a}}{({n_j}+i^{1+2\a})^4}\label{eq: MetricTrick}
\end{align}
Using the definition of $\theta_0$ given in $\eqref{eq: Counter}$ and the lower bounds $i^{1+2\a}$ and $n_j$ for the term $n_j+i^{1+2\a}$ in the denominator, the expression in the preceding display is bounded from above by constant times
\begin{align}
\sum_{1\leq i\leq 2n_{j-1}^{1/(1+2\beta)}}n_j^{-2}(\log i)^2i^{1+8D-2\beta}+\sum_{i\geq n_j^{1/(1+2\beta)}}n_j^2(\log i)^2i^{-(3+4D'+2\beta)}.\label{eq: 3terms}
\end{align}

The first term is bounded above by a multiple of $(\log n_{j-1})^{2} n_{j-1}^{(2+8D-2\beta)/(1+2\beta)}/n_j^2,$
which tends to zero faster than $n_{j-1}^{-2D'/(1+2D')}$ following from the assumptions $D<\beta<D'$ and $n_j\geq n_{j-1}^{1+4D}$. By Lemma \ref{Lemma: Appendix3} the second term of $\eqref{eq: 3terms}$ is bounded above by
\begin{align*}
n_j^{2}(\log n_j)^2n_j^{-(2+4D'+2\beta)/(1+2\beta)}\ll (\log n_j)^2n_j^{-2D'/(1+2\beta)}\ll n_j^{-2D'/(1+2D')},
\end{align*}
following from $\beta<D^*<D'$.

It remained to deal with the second term on the right hand side of $\eqref{eq: help2}$. Following from $\eqref{eq: seged}$, $\eqref{eq: UBw}$ with $\a_1=D$ and triangle inequality we have with $\PP_{\theta_0}$-probability tending to one that
\begin{align*}
\sup_{\a_1,\a_2\in[D,2D]}\| \hat\theta_{n_j,\a_1}-  \hat\theta_{n_j,\a_2}\|^2\leq 4\sup_{\a\in[D,2D]}\|\EE_{\theta_0} \hat\theta_{n_j,\a}\|^2+12 n_j^{-D/(1+2D)}\leq 4M+o(1).
\end{align*}
In Section \ref{Sec: Help2} we prove that
\begin{align}
\EE_{\theta_0} \lambda(\a\in[D,D']\,|\,X)\lesssim \exp(-c_1 n_j^{1/(1+4D)})\log n_j,\label{Ass: Help2}
\end{align}
for some constant $c_1$. Therefore by applying Markov's inequality one can obtain that the second term on the right hand side of $\eqref{eq: help2}$ has smaller rate than any polynomial with $\PP_{\theta_0}$-probability tending to one, which concludes the proof.

\subsection{Proof of assertion $\eqref{Ass: Help3}$}\label{Sec: Help3}

First we give a lower bound for the hierarchical posterior probability of the credible ball centered around the hierarchical posterior mean with radius $\hat{r}_{n_j,\gamma}$.
\begin{align*}
\Pi(\theta:\, \|\theta-\hat\theta_{n_j}\|<\hat{r}_{n_j,\gamma}|X)\geq \int_{D'
}^{2D}\Pi_{\a}(\theta:\, \|\theta-\hat\theta_{n_j}\|<\hat{r}_{n_j,\gamma}|X)\lambda(d\a|X).
\end{align*}
Then by applying triangle inequality one can observe that the right hand side of the preceding display is bounded from below by
\begin{align}
&\int_{D'}^{2D}\Pi_{\a}(\theta:\, \|\theta-\hat\theta_{n_j,\a}\|+ \|\hat\theta_{n_j}-\hat\theta_{n_j,\a}\|<\hat{r}_{n_j,\gamma}|X)\lambda(d\a|X)\geq\nonumber\\
&\quad \inf_{\a\in[D',2D]
}\Pi_{\a}(\theta:\, \|\theta-\hat\theta_{n_j,\a}\|+ \|\hat\theta_{n_j}-\hat\theta_{n_j,\a}\|<\hat{r}_{n_j,\gamma}|X) \lambda(\a\in[D',2D]|X).\label{eq: help3}
\end{align}
From assertion $\eqref{Ass: Help2}$ and by applying Markov's inequality follows that the probability $\lambda(\a\in[D',2D]|X)$ is bigger than $(1-\gamma)/(1-\gamma/2)$ with $\PP_{\theta_0}$-probability tending to one. Therefore to prove $\eqref{Ass: Help3}$ it is sufficient to show that
there exists some large enough constant $C$ such that with $\PP_{\theta_0}$-probability tending to one
\begin{align}
\inf_{\a\in[D',2D]}\Pi_{\a}(\theta:\, \|\theta-\hat\theta_{n_j,\a}\|+ \|\hat\theta_{n_j}-\hat\theta_{n_j,\a}\|<Cn_j^{-D'/(1+2D')}|X)\geq 1-\gamma/2.\label{eq: help5}
\end{align}

Following from $(5.7)$ and $(6.1)$ of \cite{SzVZ2} we have for any compact interval $[\a_1,\a_2]$ and for $n\geq (10(1+3^{1+2\a_2})/(1-\gamma))^{1+2\a_2}$ that
\begin{align}
cn^{-\a_2/(1+2\a_2)}\leq \inf_{\a\in[\a_1,\a_2]}r_{n,\gamma}(\a)\leq \sup_{\a\in[\a_1,\a_2]}r_{n,\gamma}(\a)\leq Cn^{-\a_1/(1+2\a_1)},\label{eq: BoundsR}
\end{align}
with
\begin{align*}
c= (1+3^{1+2\a_2})^{-1/2}/\sqrt{2},\quad\text{and}\quad C=\sqrt{3+2/\a_1}.
\end{align*}
Therefore one can observe that for a fixed hyper-parameter $\a\geq D'$ the radius of the $1-\gamma/2$-credible ball is bounded above by a multiple of $n_j^{-\a/(1+2\a)}\leq n_j^{-D'/(1+2D')}.$ Together with $\eqref{eq: help4}$ this concludes the proof.

\subsection{Proof of assertion $\eqref{Ass: Help2}$}\label{Sec: Help2}
In the proof of Theorem~\ref{Thm: EBCounter} we have shown that the likelihood function is monotone increasing on the interval $[D,2D]\supset[D,D'+1/(\log n_j)]$. Then by replacing $\underline\a_n$ with $D'$ and taking $p=0$ in the second paragraph of the proof of Theorem 2.5 of \cite{KSzVZ} we get that
\begin{align}
\EE_{\theta_0} \lambda(\a\in[D,D']|X)\leq \frac{\exp\big(-K{n_j}^{1/(1+2D')}/(1+2D')\big)}{\lambda(\a\in[D'+1/(2\log n_j),D'+1/(\log n_j)])}.\label{eq: hyperpost}
\end{align}
From the definition of $D'=D'(n_j)$ follows that the denominator on the right hand side of $\eqref{eq: hyperpost}$
is at least $c/(2\log n_j)$. Therefore the right hand side of $\eqref{eq: hyperpost}$ is bounded above by a multiple of $\exp(-K_1{n_j}^{1/(1+4D)})\log n_j$.

\section{Proof of Theorem \ref{thm: Coverage}}\label{sec: Coverage}

As a first step we investigate the behaviour of the new estimator of the hyper-parameter $\tilde\a_n$. We introduce the notation
\begin{align}
\underline\a_n&=\inf\{\a\geq D:\, B_{n,k_n}^{2}(\a;\theta_0)\geq (C_1^2/2) n^{-\frac{2\a}{1+2\a}}\}\wedge ((2D-C_0/\log n)\vee D),\label{def: LB}\\
\overline\a_n&=\inf\{\a\geq D:\, B_{n,k_n}^{2}(\a;\theta_0)\geq  2C_1^2n^{-\frac{2\a}{1+2\a}}\}\wedge ((2D-C_0/\log n)\vee D),\label{def: UB}
\end{align}
where $B_{n,k_n}^{2}(\a;\theta_0)$ is defined in $\eqref{def: Biask}$. The next lemma says that with high probability the estimator $\tilde\a_n$ is going to be in the interval $[\underline\a_n,\overline\a_n]$.

\begin{lemma}\label{Lemma: estimator}
For every positive $\gamma$ and $C_1$ and the constant $C_0$ defined in $\eqref{def: C0}$ we have
\begin{align}
\liminf_{n\rightarrow\infty}\inf_{\theta_0\in S^{D}(M)} \PP_{\theta_0}(\underline\alpha_n<\tilde\alpha_n<\overline\alpha_n)\geq 1-\gamma/2.\label{eq: convergence}
\end{align}
Furthermore for all $\beta\in[D,2D]$
\begin{align}
\inf_{\theta_0\in S^{\beta}(M)}\underline\a_n> \beta-[(1/2+\beta)\log(2M/C_1^2)\vee C_0]/\log n\label{eq: LB}
\end{align}
\end{lemma}
\begin{proof}
See Section \ref{sec: estimator}.
\end{proof}

Now we are ready to deal with the honest coverage assertion $\eqref{eq: PostCov}$. From Lemma \ref{Lemma: estimator} we have
\begin{align*}
\inf_{\theta_0\in S^{D}(M)}\PP_{\theta_0}(\theta_0\in \hat{C}_n^R(L))\geq \inf_{\theta_0\in S^{D}(M)}\PP_{\theta_0}\big(\sup_{\underline\a_n\leq\a\leq\overline\a_n}\|\hat{\theta}_{n,\alpha}-\theta_0\|\leq L \inf_{\underline\a_n\leq\a\leq\overline\a_n}r_{n,\gamma}(\a)\big)-\gamma/2+o(1).
\end{align*}
Let us denote by $B_n(\a;\theta_0)$ the bias $\EE_{\theta_0}\hat{\theta}_{n,\a}-\theta_0$ and by $W_n(\a)$ the centered posterior mean $\hat{\theta}_{n,\a}-\EE_{\theta_0} \hat{\theta}_{n,\a}$. From triangle inequality follows that for $\eqref{eq: PostCov}$ it suffices to prove
\begin{align*}
\inf_{\theta_0\in S^{D}(M)}\PP_{\theta_0}\big(\sup_{\underline\a_n\leq\a\leq\overline\a_n}\|W_n(\a)\|\leq
\inf_{\underline\a_n\leq\a\leq\overline\a_n} Lr_{n,\gamma}(\alpha) -\sup_{\underline\a_n\leq\a\leq\overline\a_n}\|B_n(\alpha;\theta_0)\|\big)\rightarrow 1.
\end{align*}

To establish the preceding convergence we show that
\begin{align}
\inf_{\a\in[\underline\a_n,\overline\a_n]}r_{n,\gamma}(\a)>(\sqrt{8(1+3^{1+4D})})^{-1} n^{-\underline\a_n/(1+2\underline\a_n)},\label{eq: part1}\\
\sup_{\a\in[\underline\a_n,\overline\a_n]}\|B_n(\a;\theta_0)\|\leq \sqrt{2(C_1^2+ M)} n^{-\underline\a_n/(1+2\underline\a_n)},\label{eq: part2}\\
\inf_{\theta\in S^{D}(M)} \PP_{\theta_0}\Big(\sup_{\underline\a_n\leq\a\leq\overline\a_n}\|W_n(\a)\|\leq \sqrt{6} n^{-\underline\a_n/(1+2\underline\a_n)} \Big)\rightarrow1.\label{eq: part3}
\end{align}

Assertion $\eqref{eq: part3}$ follows from $\eqref{eq: UBw}$ with $\a_1=\underline\a_n$. For assertion $\eqref{eq: part1}$ following from $\eqref{eq: BoundsR}$ it suffices to prove that
\begin{align}
n^{-\underline\a_n/(1+2\underline\a_n)}\leq2 n^{-\overline\a_n/(1+2\overline\a_n)}.\label{eq: help6}
\end{align}
If $\underline\a_n=2D-C_0/\log n$ or $\overline\a_n=D$, then necessarily $\overline\a_n=2D-C_0/\log n$ or $\underline\a_n=D$ holds, respectively and therefore $n^{-\underline\a_n/(1+2\underline\a_n)}=n^{-\overline\a_n/(1+2\overline\a_n)}$. For $\underline\a_n<2D-C_0/\log n$ and $D<\overline\a_n$ we have
\begin{align*}
(C_1^2/2) n^{-2\underline\a_n/(1+2\underline\a_n)}\leq B_{n,k_n}^2(\underline\a_n;\theta_0)
\leq B_{n,k_n}^2(\overline\a_n;\theta_0)\leq 2C_1^2 n^{-2\overline\a_n/(1+2\overline\a_n)}.
\end{align*}

To prove assertion $\eqref{eq: part2}$ we divide the sum in $\|B_n(\a;\theta_0)\|^2$ into two parts, from one to $k_n=n^{1/(1/2+2D)}$ and from $k_n+1$ to infinity.
From $\eqref{eq: BiasLK}$ we get that the second sum is bounded above by $Mn^{-4D/(1+4D)}$. To bound the sum from one to $k_n$ we distinguish two cases. For $\overline\a_n>D$ we have $\|B_{n,k_n}(\overline\a_n,\theta_0)\|^2\leq 2C_1^2n^{-2\overline\a_n/(1+2\overline\a_n)}$ and for $\overline\a_n=D$
\begin{align}
\|B_{n,k_n}(D;\theta_0)\|^2= \sum_{i=1}^{k_n}\frac{i^{2+4D}\theta_{0,i}^2}{(i^{1+2D}+n)^2}
&\leq \sum_{i=1}^{n^{1/(1+2D)}}\frac{i^{2+2D}i^{2D}\theta_{0,i}^2}{n^2}+ \sum_{i=n^{1/(1+2D)}}^{k_n}\frac{i^{2D}\theta_{0,i}^2}{i^{2D}}\nonumber\\
&\leq n^{-\frac{2D}{1+2D}}\sum_{i=1}^{\infty}i^{2D}\theta_{0,i}^2\leq M n^{-\frac{2D}{1+2D}}.\label{eq: UBbiaskn}
\end{align}

Finally we prove assertion $\eqref{eq: adaptivity}$. Since $\tilde{\a}_n\in[\underline\a_n,\overline\a_n]$ with probability bigger than $1-\gamma/2$, from \eqref{eq: BoundsR} follows that with probability bigger than $1-\gamma/2$ the radius of the empirical Bayes credible ball is bounded above by $(3+2/D)^{1/2}n^{-\underline\a_n/(1+2\underline\a_n)}$. Then by applying assertion $\eqref{eq: LB}$ and Lemma \ref{Lemma: Appendix4} we can conclude the proof.

\section{Proof of Lemma \ref{Lemma: estimator}}\label{sec: estimator}

First we deal with assertion $\eqref{eq: convergence}$ and show separately that both $\overline\a_n\geq \tilde\a_n$ and $\underline\a_n\leq \tilde\a_n$ hold with probability bigger than $1-\gamma/4$ uniformly over $S^{D}(M)$.

We start with the upper bound $\overline\a_n\geq \tilde\a_n$. Following from the definition of $\tilde\a_n$ we have to deal only with the case $\overline\a_n<2D-C_0/\log n$, where
\begin{align}
B_{n,k_n}^2(\overline\a_n;\theta_0)\geq 2C_1^2n^{-2\overline\a_n/(1+2\overline\a_n)}\label{eq: IneqAssump}
\end{align}
holds. Then it is sufficient to prove that for $\a\leq\overline\a_n$ and $n\geq (M\vee4C_1^2)^{(1+2D)(1+4D)}$
\begin{align}
\var_0 \big(\hat{B}_{n,k_n}^2(\a,\theta_0)\big)\leq 5n^{-8D/(1+4D)}\label{eq: varBn},
\end{align}
because by $\eqref{eq: IneqAssump}$ and $\eqref{eq: varBn}$
\begin{align*}
\hat{B}_{n,k_n}^2(\overline\a_n)-C_1^2n^{-\frac{2\overline\a_n}{1+2\overline\a_n}}&\geq C_1^2n^{-\frac{2\overline\a_n}{1+2\overline\a_n}}+\hat{B}_{n,k_n}^2(\overline\a_n)-B_{n,k_n}^2(\overline\a_n;\theta_0)\\
&>
C_1^2n^{-\frac{2\overline\a_n}{1+2\overline\a_n}}-\sqrt{20/\gamma} n^{-\frac{4D}{1+4D}}
\end{align*}
holds with probability bigger than $1-\gamma/4$ following from $\eqref{def: Biask}$ and Chebyshev's inequality. From Lemma \ref{Lemma: Appendix4} and $\overline\a_n\leq 2D-C_0/\log n$ one can see that the right hand side of the preceding display is  bounded below by $(C_1^2e^{2C_0/(1+4D)^2}-\sqrt{20/\gamma})  n^{-4D/(1+4D)}$, which is positive following from the definition of $C_0$ and therefore $\overline\a_n\geq\tilde\a_n$.

To show assertion $\eqref{eq: varBn}$ we note that $\var_{\theta_0} X_i^2=2\theta_{0,i}^2/n+4/n^2$ and $B_{n,k_n}^2(\a;\theta_0)$ is monotonically increasing, hence
\begin{align}
\var_{\theta_0}\big(\hat{B}_{n,k_n}^2(\a)\big)
=\sum_{i=1}^{k_n}\frac{i^{4+8\a}}{(i^{1+2\a}+n)^4}
\Big(\frac{2\theta_{0,i}^2}{n}+\frac{4}{n^2}\Big)
\leq \frac{2B_{n,k_n}^2(\overline\a_n;\theta_0)}{n}+\frac{4k_n}{n^2},\label{varBnk}
\end{align}
for all $\a\leq\overline\a_n$.
By the choice $k_n=n^{2/(1+4D)}$ the second term on the right hand side of $\eqref{varBnk}$ is bounded above by $4n^{-8D/(1+4D)}$. The first term of $\eqref{varBnk}$ for $\overline\a_n>D$ is bounded by a multiple of $n^{1/(1+2\overline\a_n)-2}$ and for $\overline\a_n=D$ it is bounded by $Mn^{-(1+4D)/(1+2D)}$ following from \eqref{eq: UBbiaskn}. Since both of the preceding rates are faster than $n^{-8D/(1+4D)}$ this concludes the proof of assertion $\eqref{eq: varBn}$.

Next we deal with the lower bound $\tilde\a_n\geq\underline\a_n$, which holds trivially for $\underline\a_n=D$. Assume that $\underline\a_n>D$ and denote the set of parameters satisfying this inequality by $\Theta_n\subset S^{D}(M)$.
From triangle inequality we have
\begin{align}
\sup_{\alpha\in[D,\underline\a_n]} \hat{B}_{n,k_n}^2(\a)-C_1^2n^{-\frac{2\a}{1+2\a}}&\leq \sup_{\alpha\in[D,\underline\a_n]}\big(|\hat{B}_{n,k_n}^2(\a)-B_{n,k_n}^2(\a;\theta_0)|-(C_1^2/2)n^{-\frac{2\a}{1+2\a}}\big)\nonumber\\
&\quad+\sup_{\alpha\in[D,\underline\a_n]}\big(B_{n,k_n}^2(\a;\theta_0)-(C_1^2/2)n^{-\frac{2\a}{1+2\a}}\big).\label{eq: help01}
\end{align}
For $\theta_0\in\Theta_n$ following from the definition of $\underline\a_n$ the second term on the right hand side is non-positive. Next we show that
\begin{align}
\sup_{\alpha\in[D,\underline\a_n]}\big(|\hat{B}_{n,k_n}^2(\a)-B_{n,k_n}^2(\a;\theta_0)|-(C_1^2/2)n^{-\frac{2\a}{1+2\a}}\big)\leq (24/\gamma)n^{-4D/(1+4D)}-(C_1^2/2)n^{-\frac{2\underline\a_n}{1+2\underline\a_n}},\label{eq: help7}
\end{align}
with probability bigger $1-\gamma/4$. Then by Lemma \ref{Lemma: Appendix4} one can obtain that the right hand side of $\eqref{eq: help7}$ and hence the right hand side of $\eqref{eq: help01}$ is bounded above by $[24/\gamma-(C_1^2/2)e^{2C_0/(1+4D)^2}]n^{-4D/(1+4D)}<0$. Therefore with probability bigger than $1-\gamma/4$ the lower bound $\underline\a_n\leq\tilde{\a}_n$ holds.

To prove $\eqref{eq: help7}$ following from Markov's inequality it suffices to show that for large enough $n$ (depending only on  $D,C_1$ and $\gamma$)
\begin{align*}
\sup_{\theta_0\in\Theta_n}\EE_{\theta_0}
\sup_{\a\in[D,\underline\a_n]}|\hat{B}_{n,k_n}^2(\a)-B_{n,k_n}^2(\a;\theta_0)|\leq 6n^{-4D/(1+4D)}.
\end{align*}
By Corollary 2.2.5 \cite{vdVW} applied with $\psi(x)=x^2$ the preceding inequality holds if
\begin{align}
n^{8D/(1+4D)}\sup_{\theta_0\in\Theta_n}\sup_{\a\in[D,\underline\a_n]}\var_{\theta_0}\big(\hat{B}_{n,k_n}^2(\a)\big)\leq 5\label{eq: varconv}\\
\sup_{\theta_0\in\Theta_n}\int_0^{diam_n}\sqrt{N(\eps,[D,\underline\a_n],d_n)}d\eps\rightarrow 0,\label{eq: covnumbconv}
\end{align}
where $d_n$ is the semimetric defined by
\begin{align*}
d_n^2(\a_1,\a_2)=n^{8D/(1+4D)}\var_{\theta_0}\big(\hat{B}_{n,k_n}^2(\a_1)
-\hat{B}_{n,k_n}^2(\a_2)\big),
\end{align*}
$N(\eps,A,d)$ is the covering number of the set $A$ with
$\eps$-balls relative to the semimetric $d$ and $diam_n$ is the
diameter of the interval $[D,\underline\a_n]$ relative to $d_n$.

The first assertion $\eqref{eq: varconv}$ follows immediately from $\eqref{eq: varBn}$. From triangle inequality one can observe that the diameter $diam_n$ is bounded above by $2\sqrt{5}$. Furthermore for any $\theta_0\in\Theta_n$ and $\a\in[D,\underline\a_n]$
\begin{align*}
\sup_{\theta_0\in\Theta_n}B_{n,k_n}^2(\a;\theta_0)\leq2C_1^2n^{-2\a/(1+2\a)}\leq 2C_1^2 n^{-2D/(1+2D)}.
\end{align*}
 Therefore from Lemma \ref{Lemma: metric} follows that for $\underline\a_n\geq\a_2>\a_1$ the semi-metric satisfies $d_n(\a_1,\a_2)\lesssim \log (n)n^{-1/[2(1+4D)(1+2D)]}|\a_1-\a_2|$. Then the covering number is bounded above by
\begin{align*}
N(\eps,[D,\underline\a_n],d_n)\lesssim\log (n)n^{-1/[2(1+4D)(1+2D)]}/\eps,
\end{align*}
hence the integral
\begin{align*}
\sup_{\theta_0\in\Theta_n}\int_0^{2\sqrt{5}}\sqrt{N(\eps,[D,\underline\a_n],d_n)}d\eps \rightarrow0.
\end{align*}

It remained to prove assertion $\eqref{eq: LB}$. For $C>0$ we have
\begin{align*}
\sup_{\theta\in S^{\beta}(M)}B_{n,k_n}^2(\beta-\frac{C}{\log n},\theta_0)
&\leq \sum_{1\leq i\leq n^{\frac{1}{1+2\beta-2C/\log n}}}\frac{i^{2+2\beta-4C/\log n}i^{2\beta}\theta_{0,i}^2}{n^2}
+ \sum_{i>n^{\frac{1}{1+2\beta-2C/\log n}}}\frac{i^{2\beta}\theta_{0,i}^2}{i^{2\beta}}\\
&\leq \|\theta_{0}\|_{\beta}^2n^{-\frac{2\beta}{1+2\beta-2C/\log n}} < M e^{-2C/(1+2\beta)}n^{-\frac{2\beta-2C/\log n}{1+2\beta-2C/\log n}}.
\end{align*}
For $e^{C}\geq (2M/C_1^2)^{1/2+\beta}$ the constant multiplier on the right hand side is bounded above by $C_1^2/2$. Since the function $B_{n,k_n}^2(\a;\theta_0)$ is monotonically increasing and $f_n(\a)=n^{-2\a/(1+2\a)}$ is strictly monotonically decreasing one can conclude that for $\a<\beta-C_1/\log n$ the inequality $B_{n,k_n}^2(\a;\theta_0)<(C_1^2/2)n^{-2\a/(1+2\a)}$ holds and hence $\underline\a_n\geq\beta-[(1/2+\beta)\log(2M/C_1^2)\vee C_0]/\log n$.

\begin{lemma}\label{Lemma: metric}

For any $D\leq\alpha_1<\alpha_2\leq 2D$ we
\begin{align*}
\sup_{\theta_0\in S^D(M)}\var_{\theta_0}\big(\hat{B}_{n,k_n}^2(\a_1)-\hat{B}_{n,k_n}^2(\a_2)\big)\lesssim
(\a_2-\a_1)^2(\log n)^2(B_{n,k_n}^2(\a_2;\theta_0)/n+n^{-\frac{1+4D}{1+2D}})
\end{align*}
\end{lemma}

\begin{proof}
The left hand side of the inequality in the lemma is equal to
\begin{align}
\sum_{i=1}^{k_n}\Big(\frac{i^{2+4\a_1}}{(i^{1+2\a_1}+n)^2}-\frac{i^{2+4\a_2}}{(i^{1+2\a_2}+n)^2}\Big)^2
\var_{\theta_0} X_i^2.\label{eq: help}
\end{align}
Furthermore the derivative of the function $f_i(\a)=i^{2+4\a}(i^{1+2\a}+n)^{-2}$ satisfies
\begin{align*}
|f_i(\a)'|\leq 4n\log(i)i^{2+4\a}(i^{1+2\a}+n)^{-3}.
\end{align*}

Applying $\var_{\theta_0}X_i^2=2\theta_{0,i}^2/n+4/n^2$, the preceding upper bound for $|f_i(\a)'|$ and Cauchy-Schwarz inequality, similarly to $\eqref{eq: MetricTrick}$ we get that
$\eqref{eq: help}$ is bounded above by a multiple of
\begin{align*}
(\a_2-\a_1)^2& \log^2(n)\sup_{\a\in[\a_1,\a_2]}\frac{1}{n}\sum_{i=1}^{k_n}\frac{n^2i^{4+8\a}\theta_{0,i}^2}{(i^{1+2\a}+n)^6}\\
&\quad+ (\a_2-\a_1)^2 \log^2(n)\sup_{\a\in[\a_1,\a_2]}\frac{1}{n^2}\sum_{i=1}^{k_n}\frac{n^2i^{4+8\a}}{(i^{1+2\a}+n)^6}.
\end{align*}
One can obtain that the first term on the right hand side is bounded above by constant times $(\a_2-\a_1)^2 \log^2(n) \sup_{\a\in[\a_1,\a_2]}B_{n,k_n}^2(\a;\theta_0)/n$,
 while the second term by a multiple of $(\a_2-\a_1)^2 (\log n)^2n^{-2+1/(1+2\a)}$. The assertion of the lemma follows from the monotonically increasing property of $B_{n,k_n}^2(\a;\theta_0)$.
\end{proof}

\section{Appendix}\label{sec: Appendix}
In this section we collected the proof of Theorem~\ref{Thm: EBCounter} (which is based on the proof of Theorem 3.1 of \cite{SzVZ2}), and additional auxiliary lemmas.

\subsection{Proof of Theorem~\ref{Thm: EBCounter}}\label{sec: EBCounter}
First of all it is easy to see that for given $M>0$ and for arbitrary parameter $\beta'\in[D,\beta)$ we can choose the constant $K$ small enough (depending only on $M$ and $\beta-\beta'$) that $\theta_0\in S^{\beta'}(M)$.

Then following from \cite{KSzVZ} and \cite{SzVZ} we define the function
\begin{align}
h_n(\alpha;\theta_0)= \frac{1+2\alpha}{n^{1/(1+2\alpha)}\log n}\sum_{i=1}^{\infty}\frac{n^2i^{1+2\alpha}\log (i)\theta_{0,i}^2}{(i^{1+2\alpha}+n)^2},\label{eq: hn(alpha)}
\end{align}
and introduce the variable
\begin{align*}
\underline{\alpha}_{0,n}&=\inf\{\alpha>0: h_n(\alpha;\theta_0)>1/16\}\wedge\sqrt{\log n}.
\end{align*}
From Theorem 2.2 of \cite{KSzVZ} with $p=0$ follows that with probability tending to one the likelihood function is monotonically increasing for $\alpha\leq\underline\a_{0,n}+1/\log n$, hence the maximum in $[D,2D]$ is taken for some $\alpha\geq (\underline\a_{0,n}\vee D)\wedge 2D$. We show below that $\underline\a_{0,n_j}>2D$, hence with $\PP_{\theta_0}$-probability tending to one $\hat\a_{n_j}=2D$.

Using the notation $\eqref{def: BiasVar}$ we have that
$\th_0\in\hat C_n^{E}(L)$ if and only if $\|\hat\th_{\hat\a_n}-\th_0\|\le L r_{n,\gamma}(\hat\a_n)$,
which implies that $\bigl\|B_n(\hat\a_n;\th_0)\bigr\|\le Lr_{n,\gamma}(\hat \a_n)+\bigl\|W_n(\hat\a_n)\bigr\|$.
Combined with $\PP_{\theta_0}(\hat\a_{n_j}=2D)\rightarrow1$ it follows that $\PP_{\th_0}\bigl(\th_0\in\hat C_n^{E}(L)\bigr)$
is bounded above by
\begin{align}
\PP_{\th_0}\Big(\|B_n(2D,\th_0)\bigr\|\leq L
r_{n,\gamma}(2D)+\bigl\|W_n(2D)\bigr\|\Big) + o(1).
\label{eq: CounterUBcoverage}
\end{align}

Assertions $\eqref{eq: UBw}$ with $\a_1=2D$, $\eqref{eq: BiasLB}$ with $D'=2D$ and $\eqref{eq: BoundsR}$ with $\a_1=2D$ show that
\begin{align*}
\sup_{\th_0\in\ell_2}r_{n,\gamma}(2D) &\lesssim
n^{-2D/(1+4D)},\\
\|B_n(2D;\theta_0)\|^2\gtrsim n^{-2\beta/(1+2\beta)}\\
\inf_{\th_0 \in \ell_2}\PP_{\theta_0}\Big(\big\|W(2D)\big\|
&\le C n^{-2D/(1+4D)}\Big) \rightarrow 1.
\end{align*}
Thus we deduce  that the expression to the left of the
inequality sign in \eqref{eq: CounterUBcoverage}
is of larger order than the expression to the right, whence the probability
tends to zero along the subsequence $n_j$.

Finally we prove the claim that $\lb\a_{0,n_j}\ge2D$, by showing that
 $h_{n_j}(\a;\th_0)<1/16$ for all $\a<2D$. Let $N_j=n_j^{1/(1+2\beta)}$ then we have
$h_{n_j}(\a;\th_0) \le A_1+A_2+A_3$ for
\begin{align*}
A_1&=\frac{1+2\a}{n_j^{1/(1+2\a)}\log n_{j}}\sum_{i \le 2N_{j-1}}Ki^{2\a-2\b}\log i,\\
A_2&= \frac{1+2\a}{n_j^{1/(1+2\a)}\log n_{j}}
\sum_{N_j\le i < 2N_{j}}\frac{n_j^2i^{2\a-2\beta}(\log i)K2^{-1-2\b}}{(i^{1+2\a}+n_j)^2},\\
A_3&=\frac{1+2\a}{n_j^{1/(1+2\a)}\log n_{j}}\sum_{i \ge N_{j+1}}Kn_{j}^2 i^{-2-2\a-2\b}(\log i).
\end{align*}
For $\a< 2D$, so that $i^{2\a-2\b}\le i^{2D}$,
\begin{align*}
A_1\lesssim n_j^{-\frac{1}{1+2\a}} N_{j-1}^{1+2D}
\lesssim n_{j-1}n_j^{-\frac{1}{1+4D}}\rightarrow0,
\end{align*}
since $n_{j-1}^{1+4D}\ll n_j$. By Lemma~\ref{Lemma: Appendix3} the third term
satisfies
\[
A_3 \lle \frac{\log N_{j+1}}{\log n_j}n_j^{\frac{1+4\a}{1+2\a}}N_{j+1}^{-(1 + 2\a + 2\b )}\lle \frac{(\log n_{j+1})/n_{j+1}}{(\log n_j)/n_j}n_j^{\frac{2\a}{1+2\a}}n_{j+1}^{-\frac{2\a}{1+2\beta}}.
\]
Because $n_{j}^{1+4D}\ll n_{j+1}$ (and $f(x)=(\log x)/x=0$ is monotone decreasing for $x\geq e$), this is also easily seen to vanish as $j \ra \infty$.
The term $i^{1+2\a}+n_j$ in the denominator of the sum in $A_2$ can be bounded
below both by $i^{1+2\a}$ and by $n_j$, and there are at most $N_j$ terms in the sum.
This shows that
\begin{align*}
A_2 &\lesssim \frac{n_j^{-1/(1+2\a)}}{\log n_{j}}
N_j \Bigl(\frac{n_j^2}{N_j^{1+2\a}}\wedge (2N_j)^{1+2\a}\Bigr)\log (2N_j)N_j^{-1-2\b}K\\
& \lle
 K \Bigl(n_j^{\frac{1+2\b - 2\a}{1+2\b}- \frac{1}{1+2\a}}
\wedge n_j^{\frac{1+2\a - 2\b}{1+2\b}- \frac{1}{1+2\a}}\Bigr).
\end{align*}
The exponents of $n_j$ in both terms in the minimum are equal to 0 at $\a=\b$.
For $\a\ge\b$ the first exponent is negative, whereas the second exponent is increasing
in $\a$ and hence negative for $\a<\b$. It follows that $A_2 \lle K$.

Putting things together we see that
$\limsup_{j\ra\infty} \sup_{\a\le 2D}h_{n_j}(\a; \th_0)$
can be made arbitrarily small by choosing $K$ sufficiently small.

\subsection{Auxiliary lemmas}
We collected the auxiliary lemmas in this section.

\begin{lemma}[Lemma 10.4 of \cite{SzVZ2}]\label{Lemma: Appendix3}
For $k>0$, $m\geq 0$, and $N\geq e^{2m/k}$,
\begin{align*}
\sum_{i>N}i^{-1-k}(\log i)^{m}\leq (1/N+2/k)(\log N)^{m}N^{-k}.
\end{align*}
\end{lemma}

\begin{lemma}\label{Lemma: Appendix4}
For the function $f_n(\a)=n^{-2\a/(1+2\a)}$, $\a\in[D,2D]$, $K>0$ and $n\geq e^{K/4}$ we have
\begin{align*}
e^{2K/(1+4D)^2}f_n(\a)\leq f_n(\a-K/\log n)\leq  e^{2K/(1/2+2D)^2}f_n(\a),\\
 e^{-2K/(1+2D)^2}f_n(\a)\leq f_n(\a+K/\log n)\leq e^{-2K/(1+4D)^2}f_n(\a).
\end{align*}
\end{lemma}
\begin{proof}
Since $(\log f_n(\a))'=-2(\log n)/(1+2\a)^2$ we have
\begin{align*}
\log\frac{f_n(\a-K/\log n)}{f_n(\a)}\geq (-K/\log n) (-2\log n)/(1+2\a)^2\geq K/(1+4D)^2,\\
\log\frac{f_n(\a-K/\log n)}{f_n(\a)}\leq K/(1+2\a-2K/\log n)^2\leq K/(1/2+2D)^2,
\end{align*}
where the second inequality in the second line holds for $n\geq e^{K/4}$. The second part of the lemma follows similarly.
\end{proof}

\bibliographystyle{acm}
\bibliography{credibility}
\end{document}